\documentclass[11pt]{amsart}

\usepackage{amsmath,amssymb,amsthm,mathtools}
\usepackage{enumitem}
\usepackage{graphicx}
\usepackage{hyperref}
\hypersetup{hidelinks}

\numberwithin{equation}{section}
\newtheorem{theorem}{Theorem}[section]
\newtheorem{proposition}[theorem]{Proposition}
\newtheorem{lemma}[theorem]{Lemma}
\newtheorem{corollary}[theorem]{Corollary}

\theoremstyle{definition}
\newtheorem{definition}[theorem]{Definition}
\newtheorem{example}[theorem]{Example}

\theoremstyle{remark}
\newtheorem{remark}[theorem]{Remark}

\DeclareMathOperator{\Mag}{Mag}
\newcommand{\R}{\mathbb{R}}
\newcommand{\one}{\mathbf{1}}
\newcommand{\cA}{\mathcal{A}}
\newcommand{\cB}{\mathcal{B}}
\newcommand{\Zmat}{Z}

\title[Small-scale magnitude below one]{Small-scale magnitude below one for cyclic two-chunk finite metric spaces}

\author{Tsubasa Kamiyama}
\date{}
\subjclass[2020]{51F99, 05C12, 15A18}
\keywords{finite metric spaces, magnitude, distance matrices, small-scale limit, one-point property}

\begin{document}

\begin{abstract}
Motivated by the small-scale viewpoint of Roff and Yoshinaga, we study finite metric spaces whose scaled copies collapse to a single point while their magnitude remembers how the collapse takes place.  The limit metric space is geometrically indistinguishable from a point, but the magnitude function can detect differences in the path of collapse.  We introduce a four-parameter family of cyclic two-chunk finite metric spaces, compute their magnitude explicitly, and use the formula to construct balanced examples whose small-scale magnitude is less than one.  In particular, we exhibit a twelve-point finite metric space satisfying
\(
\lim_{t\to 0+}\Mag(tX)=44/59<1.
\)
The guiding question and the terminology around the one-point property come from Leinster's magnitude of finite metric spaces and from Roff--Yoshinaga's work on small-scale limits.
\end{abstract}

\maketitle

\section{Introduction}

Magnitude was introduced by Leinster as a size invariant of enriched
categories and, in particular, of finite metric spaces \cite{Leinster2013}.
For a finite metric space \(X\), its scaled copy \(tX\) is obtained by
multiplying all distances by \(t>0\).  As \(t\to 0^+\), the spaces \(tX\)
collapse to a one-point metric space.  The one-point property asks whether
this geometric collapse is reflected by magnitude, namely whether
\begin{equation}\label{eq:one-point-property}
\lim_{t\to 0^+}\Mag(tX)=1.
\end{equation}

The small-scale behaviour of magnitude was recently studied systematically
by Roff and Yoshinaga \cite{RoffYoshinaga2025}.  Their work gives a sharp
picture of the one-point property for finite metric spaces.  They proved
that, for each fixed cardinality, the space of finite metric spaces contains
a dense open subset on which the one-point property holds.  In this sense,
the one-point property is generic.  They also showed that this generic
behaviour is best possible: every metric space with at most four points has
the one-point property, while the property already fails for a certain
five-point metric space.  In their example, the space is obtained from a
two-point space of distance \(4/3\) and a three-point equilateral space of
distance \(2\), and its small-scale magnitude is \(7/6\).

An important insight of Roff and Yoshinaga's work is that failure of the
one-point property is not merely a technical pathology.  Although such
failure is non-generic, it can be quantitatively large: the small-scale
limit of magnitude can be made to take arbitrary prescribed real values
greater than one.  Thus their results reveal a striking contrast.  Magnitude
is generically stable under collapse to a point, but, in exceptional cases,
the small-scale limit can retain delicate information about the manner in
which the original finite metric space collapses.

The present note continues this line of investigation by showing that the
exceptional behaviour is not confined to the side above one.  We construct
explicit finite metric spaces whose small-scale magnitude is strictly less
than one.  More precisely, we introduce a four-parameter family of cyclic
two-chunk finite metric spaces, compute their magnitude explicitly, and
exhibit a twelve-point example satisfying
\[
\lim_{t\to 0^+}\Mag(tX)=\frac{44}{59}<1.
\]
Consequently, a finite metric space can collapse to a point while its
magnitude converges to a value below the magnitude of a point.

Before giving the definitions, we display the motivating example.  Let \(X\) be the twelve-point finite metric space whose distance matrix is
\begingroup
\setlength{\arraycolsep}{4pt}
\renewcommand{\arraystretch}{1.18}
\begin{equation}\label{eq:intro-matrix}
\resizebox{\textwidth}{!}{$\displaystyle D=
\begin{pmatrix}
0&\frac{7}{3}&\frac{7}{3}&\frac{7}{3}&\frac{7}{3}&\frac{7}{3}&3&3&\frac{7}{6}&\frac{7}{6}&\frac{7}{6}&\frac{7}{6}\\
\frac{7}{3}&0&\frac{7}{3}&\frac{7}{3}&\frac{7}{3}&\frac{7}{3}&\frac{7}{6}&3&3&\frac{7}{6}&\frac{7}{6}&\frac{7}{6}\\
\frac{7}{3}&\frac{7}{3}&0&\frac{7}{3}&\frac{7}{3}&\frac{7}{3}&\frac{7}{6}&\frac{7}{6}&3&3&\frac{7}{6}&\frac{7}{6}\\
\frac{7}{3}&\frac{7}{3}&\frac{7}{3}&0&\frac{7}{3}&\frac{7}{3}&\frac{7}{6}&\frac{7}{6}&\frac{7}{6}&3&3&\frac{7}{6}\\
\frac{7}{3}&\frac{7}{3}&\frac{7}{3}&\frac{7}{3}&0&\frac{7}{3}&\frac{7}{6}&\frac{7}{6}&\frac{7}{6}&\frac{7}{6}&3&3\\
\frac{7}{3}&\frac{7}{3}&\frac{7}{3}&\frac{7}{3}&\frac{7}{3}&0&3&\frac{7}{6}&\frac{7}{6}&\frac{7}{6}&\frac{7}{6}&3\\
3&\frac{7}{6}&\frac{7}{6}&\frac{7}{6}&\frac{7}{6}&3&0&\frac{29}{15}&\frac{29}{15}&\frac{29}{15}&\frac{29}{15}&\frac{29}{15}\\
3&3&\frac{7}{6}&\frac{7}{6}&\frac{7}{6}&\frac{7}{6}&\frac{29}{15}&0&\frac{29}{15}&\frac{29}{15}&\frac{29}{15}&\frac{29}{15}\\
\frac{7}{6}&3&3&\frac{7}{6}&\frac{7}{6}&\frac{7}{6}&\frac{29}{15}&\frac{29}{15}&0&\frac{29}{15}&\frac{29}{15}&\frac{29}{15}\\
\frac{7}{6}&\frac{7}{6}&3&3&\frac{7}{6}&\frac{7}{6}&\frac{29}{15}&\frac{29}{15}&\frac{29}{15}&0&\frac{29}{15}&\frac{29}{15}\\
\frac{7}{6}&\frac{7}{6}&\frac{7}{6}&3&3&\frac{7}{6}&\frac{29}{15}&\frac{29}{15}&\frac{29}{15}&\frac{29}{15}&0&\frac{29}{15}\\
\frac{7}{6}&\frac{7}{6}&\frac{7}{6}&\frac{7}{6}&3&3&\frac{29}{15}&\frac{29}{15}&\frac{29}{15}&\frac{29}{15}&\frac{29}{15}&0
\end{pmatrix}$}
\end{equation}
\endgroup
The terminology used to describe this matrix will be introduced below.  It is a
\((6,2;7/3,29/15,3,7/6)\)-two-chunk cyclic metric space; equivalently, its
parameters are
\begin{equation}\label{eq:intro-parameters}
\alpha=\frac{7}{3},\qquad
\beta=\frac{29}{15},\qquad
\gamma=3,
\qquad
\delta=\frac{7}{6}.
\end{equation}
Using the closed formula proved in Proposition \ref{prop:magnitude-formula} and applying l'Hopital's rule twice as in Proposition \ref{prop:lhopital}, one obtains
\begin{equation}\label{eq:intro-limit}
\lim_{t\to 0+}\Mag(tX)=\frac{44}{59}<1.
\end{equation}
Thus this example collapses to a one-point space, but its small-scale magnitude does not converge to the magnitude of a point.

\section{Preliminaries}

\begin{definition}[Finite metric space]\label{def:finite-metric-space}
A \emph{finite metric space} is a pair \((X,d)\), where \(X\) is a finite set and
\begin{equation}\label{eq:metric-axioms}
 d\colon X\times X\longrightarrow \R_{\geq 0}
\end{equation}
 satisfies, for all \(x,y,z\in X\),
\begin{equation}\label{eq:metric-axioms-expanded}
 d(x,y)=0\Longleftrightarrow x=y,\qquad
 d(x,y)=d(y,x),\qquad
 d(x,z)\leq d(x,y)+d(y,z).
\end{equation}
\end{definition}

\begin{definition}[Magnitude of a finite metric space]\label{def:magnitude}
Let \((X,d)\) be a finite metric space with
\begin{equation}\label{eq:ordered-space}
X=\{x_1,\ldots,x_m\}.
\end{equation}
Its \emph{zeta matrix} is
\begin{equation}\label{eq:zeta-matrix}
\Zmat_X=\bigl(e^{-d(x_i,x_j)}\bigr)_{1\leq i,j\leq m}.
\end{equation}
A \emph{weighting} on \(X\) is a vector \(w=(w_1,
\ldots,w_m)^T\in \R^m\) satisfying
\begin{equation}\label{eq:weighting-equation}
\Zmat_X w=\one,
\end{equation}
where
\begin{equation}\label{eq:one-vector}
\one=(1,1,\ldots,1)^T\in \R^m.
\end{equation}
If a weighting exists, the \emph{magnitude} of \(X\) is
\begin{equation}\label{eq:magnitude-definition}
\Mag(X)=\sum_{i=1}^m w_i.
\end{equation}
When \(\Zmat_X\) is invertible, the weighting is unique and
\begin{equation}\label{eq:magnitude-inverse}
\Mag(X)=\one^T\Zmat_X^{-1}\one.
\end{equation}
\end{definition}

\begin{definition}[Scaled metric space]\label{def:scaled-metric-space}
Let \((X,d)\) be a finite metric space and let \(t>0\).  The \emph{\(t\)-scaled metric space} is
\begin{equation}\label{eq:t-scaled-space}
 tX=(X,td),
\end{equation}
where
\begin{equation}\label{eq:t-scaled-distance}
(td)(x,y)=t\,d(x,y)
\end{equation}
for all \(x,y\in X\).
\end{definition}

\begin{definition}[Distance matrix]\label{def:distance-matrix}
A real \(m\times m\) matrix \(D=(D_{ij})\) is called a \emph{distance matrix} if
\begin{equation}\label{eq:distance-matrix-conditions}
D_{ii}=0,\qquad
D_{ij}=D_{ji}>0\quad (i\neq j),
\end{equation}
and if the triangle inequalities
\begin{equation}\label{eq:matrix-triangle-inequality}
D_{ij}\leq D_{i\ell}+D_{\ell j}
\end{equation}
hold for all \(1\leq i,j,\ell\leq m\).
\end{definition}

\begin{definition}[Metric space generated by a distance matrix]\label{def:generated-space}
Let \(D=(D_{ij})\) be an \(m\times m\) distance matrix.  The finite metric space generated by \(D\) is
\begin{equation}\label{eq:generated-space}
X_D=\{x_1,\ldots,x_m\},
\end{equation}
with metric \(d_D\) defined by
\begin{equation}\label{eq:generated-distance}
d_D(x_i,x_j)=D_{ij}.
\end{equation}
\end{definition}

Before introducing the cyclic family, we fix the following notation: for
\(s\in\R\), \(\lfloor s\rfloor\) denotes the greatest integer not exceeding
\(s\).

\begin{definition}[\((n,k;\alpha,\beta,\gamma,\delta)\)-two-chunk cyclic distance matrix]\label{def:two-chunk-matrix}
Let \(n\geq 2\), let \(1\leq k\leq \lfloor n/2\rfloor\), and let
\begin{equation}\label{eq:parameters-positive}
\alpha,\beta,\gamma,\delta>0.
\end{equation}
Put
\begin{equation}\label{eq:two-block-set}
\cA=\{a_0,\ldots,a_{n-1}\},\qquad
\cB=\{b_0,\ldots,b_{n-1}\},
\end{equation}
where \(a_{i+n}=a_i\) and \(b_{i+n}=b_i\) for every integer \(i\).  The \emph{\((n,k;\alpha,\beta,\gamma,\delta)\)-two-chunk cyclic matrix} is the symmetric \(2n\times 2n\) matrix \(D_{n,k}(\alpha,\beta,\gamma,\delta)\) whose entries are
\begin{align}
D(a_i,a_j)&=
\begin{cases}
0,& i=j,\\
\alpha,& i\neq j,
\end{cases}
\label{eq:A-block-distance}\\
D(b_i,b_j)&=
\begin{cases}
0,& i=j,\\
\beta,& i\neq j,
\end{cases}
\label{eq:B-block-distance}\\
D(a_i,b_j)&=D(b_j,a_i)=
\begin{cases}
\gamma,&
\begin{aligned}[t]
&\text{if } j\equiv i+s \pmod{n}\\
&\text{for some } s\in\{0,1,\ldots,k-1\},
\end{aligned}\\
\delta,& \text{otherwise}.
\end{cases}
\label{eq:cross-block-distance}
\end{align}
If this matrix is a distance matrix in the sense of Definition \ref{def:distance-matrix}, then it is called an \emph{\((n,k;\alpha,\beta,\gamma,\delta)\)-two-chunk cyclic distance matrix}.
\end{definition}

\begin{definition}[\((n,k;\alpha,\beta,\gamma,\delta)\)-two-chunk cyclic metric space]\label{def:two-chunk-space}
The metric space generated by an \((n,k;\alpha,\beta,\gamma,\delta)\)-two-chunk cyclic distance matrix is called an \emph{\((n,k;\alpha,\beta,\gamma,\delta)\)-two-chunk cyclic metric space}.  We denote it by
\begin{equation}\label{eq:two-chunk-space-notation}
X_{n,k}(\alpha,\beta,\gamma,\delta).
\end{equation}
\end{definition}

\begin{lemma}[Triangle inequalities for the two-chunk matrix]\label{lem:metric-conditions}
For \(n\geq 3\) and \(1\leq k\leq \lfloor n/2\rfloor\), the matrix \(D_{n,k}(\alpha,\beta,\gamma,\delta)\) is a distance matrix if and only if the following inequalities hold:
\begin{align}
\max\{\alpha,\beta\}&\leq 2\gamma &&\text{if } k\geq 2,
\label{eq:metric-condition-gamma-gamma}\\
\max\{\alpha,\beta\}&\leq 2\delta,
\label{eq:metric-condition-delta-delta}\\
\max\{\alpha,\beta\}&\leq \gamma+\delta,
\label{eq:metric-condition-mixed-sum}\\
|\gamma-\delta|&\leq \min\{\alpha,\beta\}.
\label{eq:metric-condition-mixed-difference}
\end{align}
\end{lemma}

\begin{proof}
It is enough to check non-degenerate triangles.  Triangles contained entirely
in \(\cA\) have all side lengths equal to \(\alpha\), and triangles contained
entirely in \(\cB\) have all side lengths equal to \(\beta\).  Hence those
triangles automatically satisfy the triangle inequality.

Consider next a triangle with two vertices in \(\cA\) and one vertex in
\(\cB\), say \(a_i,a_{i'}\in\cA\) with \(i\neq i'\), and \(b_j\in\cB\).  One
side has length \(\alpha\).  For a vertex \(a_\ell\in\cA\), the cross-block
side \(d(a_\ell,b_j)\) has length \(\gamma\) exactly when
\begin{equation}\label{eq:gamma-congruence-A}
j\equiv \ell+s \pmod{n}
\quad\text{for some }s\in\{0,1,\ldots,k-1\},
\end{equation}
and otherwise it has length \(\delta\).  Equivalently, for fixed \(j\), the
\(\gamma\)-neighbours of \(b_j\) inside \(\cA\) are precisely
\begin{equation}\label{eq:gamma-neighbours-A}
a_j,a_{j-1},\ldots,a_{j-k+1},
\end{equation}
where subscripts are understood modulo \(n\).  Hence exactly \(k\) vertices
of \(\cA\) have distance \(\gamma\) from \(b_j\), and the remaining
\(n-k\) vertices have distance \(\delta\) from \(b_j\).  Therefore the two
cross-block side lengths are either \((\gamma,\gamma)\), \((\delta,\delta)\),
or \((\gamma,\delta)\).  Consequently the side lengths of such a triangle are
one of
\begin{equation}\label{eq:triangle-types-A}
\alpha,\gamma,\gamma;\qquad
\alpha,\delta,\delta;\qquad
\alpha,\gamma,\delta.
\end{equation}
The type \((\alpha,\gamma,\gamma)\) occurs if and only if \(k\geq2\), because
a non-degenerate triangle must use two distinct vertices of \(\cA\).  The type
\((\alpha,\delta,\delta)\) occurs because \(n-k\geq2\), which follows from
\(k\leq\lfloor n/2\rfloor\) and \(n\geq3\).  The mixed type
\((\alpha,\gamma,\delta)\) occurs because \(k\geq1\) and \(n-k\geq1\).

The same argument for a triangle with two vertices in \(\cB\) and one vertex
in \(\cA\) gives the same three possibilities with \(\alpha\) replaced by
\(\beta\):
\begin{equation}\label{eq:triangle-types-B}
\beta,\gamma,\gamma;\qquad
\beta,\delta,\delta;\qquad
\beta,\gamma,\delta.
\end{equation}
Hence the triangle inequality is equivalent to the following conditions.  For
the triples \((s,\gamma,\gamma)\), with \(s\in\{\alpha,\beta\}\), one needs
\(s\leq2\gamma\), and this condition is relevant precisely when \(k\geq2\).
For the triples \((s,\delta,\delta)\), one needs \(s\leq2\delta\).  For the
mixed triples \((s,\gamma,\delta)\), one needs
\begin{equation}\label{eq:mixed-triangle-equivalent}
s\leq\gamma+\delta,
\qquad
|\gamma-\delta|\leq s.
\end{equation}
Combining these conditions for \(s=\alpha\) and \(s=\beta\) gives exactly
\eqref{eq:metric-condition-gamma-gamma}--\eqref{eq:metric-condition-mixed-difference}.
\end{proof}

\begin{remark}\label{rem:k-symmetry}
Interchanging \(\gamma\) and \(\delta\) replaces \(k\) by \(n-k\).  Thus the convention \(1\leq k\leq \lfloor n/2\rfloor\) removes a redundant copy of the same family.
\end{remark}

\section{Magnitude formula for two-chunk cyclic spaces}

\begin{definition}[Notation for the parameters]\label{def:alpha-beta-gamma-delta}
Let
\begin{equation}\label{eq:X-two-chunk-notation}
X=X_{n,k}(\alpha,\beta,\gamma,\delta)
\end{equation}
be an \((n,k;\alpha,\beta,\gamma,\delta)\)-two-chunk cyclic metric space.  The parameters are interpreted as follows:
\begin{align}
\alpha&=d(a_i,a_j) &&(i\neq j),
\label{eq:alpha-notation}\\
\beta&=d(b_i,b_j) &&(i\neq j),
\label{eq:beta-notation}\\
\gamma&=d(a_i,b_j) &&\bigl(j\equiv i+s \pmod{n}
\text{ for some }s\in\{0,1,\ldots,k-1\}\bigr),
\label{eq:gamma-notation}\\
\delta&=d(a_i,b_j) &&\bigl(j\not\equiv i+s \pmod{n}
\text{ for every }s\in\{0,1,\ldots,k-1\}\bigr).
\label{eq:delta-notation}
\end{align}
For the scaled space \(tX\), set
\begin{align}
A(t)&=1+(n-1)e^{-t\alpha},
\label{eq:A-t-definition}\\
B(t)&=1+(n-1)e^{-t\beta},
\label{eq:B-t-definition}\\
C(t)&=k e^{-t\gamma}+(n-k)e^{-t\delta}.
\label{eq:C-t-definition}
\end{align}
\end{definition}

\begin{proposition}[Magnitude formula]\label{prop:magnitude-formula}
Let \(X=X_{n,k}(\alpha,\beta,\gamma,\delta)\) be an \((n,k;\alpha,\beta,\gamma,\delta)\)-two-chunk cyclic metric space.  If
\begin{equation}\label{eq:nonzero-denominator-finite-t}
A(t)B(t)-C(t)^2\neq 0,
\end{equation}
then
\begin{equation}\label{eq:magnitude-formula}
\Mag(tX)=\frac{n\bigl(A(t)+B(t)-2C(t)\bigr)}{A(t)B(t)-C(t)^2}.
\end{equation}
\end{proposition}

\begin{proof}
Consider a vector that is constant on each of the two blocks, and write it as
\begin{equation}\label{eq:block-weighting}
w(a_i)=w_{\cA}(t),\qquad
w(b_i)=w_{\cB}(t)
\end{equation}
for all \(i\).  For such a vector, every weighting equation attached to a point of \(\cA\) has left-hand side
\(
A(t)w_{\cA}(t)+C(t)w_{\cB}(t)
\),
and every weighting equation attached to a point of \(\cB\) has left-hand side
\(
C(t)w_{\cA}(t)+B(t)w_{\cB}(t)
\).
Thus this block-constant vector is a weighting exactly when it satisfies the \(2\times 2\) system
\begin{equation}\label{eq:block-weighting-system}
\begin{pmatrix}
A(t)&C(t)\\
C(t)&B(t)
\end{pmatrix}
\begin{pmatrix}
w_{\cA}(t)\\
w_{\cB}(t)
\end{pmatrix}
=
\begin{pmatrix}
1\\
1
\end{pmatrix}.
\end{equation}
Solving \eqref{eq:block-weighting-system} gives
\begin{align}
w_{\cA}(t)&=\frac{B(t)-C(t)}{A(t)B(t)-C(t)^2},
\label{eq:w-A}\\
w_{\cB}(t)&=\frac{A(t)-C(t)}{A(t)B(t)-C(t)^2}.
\label{eq:w-B}
\end{align}
Therefore
\begin{equation}\label{eq:mag-sum-weights}
\Mag(tX)=n\bigl(w_{\cA}(t)+w_{\cB}(t)\bigr)
=\frac{n\bigl(A(t)+B(t)-2C(t)\bigr)}{A(t)B(t)-C(t)^2}.
\end{equation}
\end{proof}

\begin{definition}[Numerator and denominator]\label{def:N-D}
For an \((n,k;\alpha,\beta,\gamma,\delta)\)-two-chunk cyclic metric space, define
\begin{align}
N(t)&=n\bigl(A(t)+B(t)-2C(t)\bigr),
\label{eq:N-t-definition}\\
D(t)&=A(t)B(t)-C(t)^2.
\label{eq:D-t-definition}
\end{align}
Thus, whenever \(D(t)\neq 0\),
\begin{equation}\label{eq:mag-N-D}
\Mag(tX)=\frac{N(t)}{D(t)}.
\end{equation}
\end{definition}

\begin{proposition}[Vanishing at the collapsed scale]\label{prop:N-D-vanish}
For every metric space \(X_{n,k}(\alpha,\beta,\gamma,\delta)\) in this family,
\begin{align}
\lim_{t\to 0+}N(t)&=0,
\label{eq:N-zero-limit}\\
\lim_{t\to 0+}D(t)&=0.
\label{eq:D-zero-limit}
\end{align}
\end{proposition}

\begin{proof}
Since
\begin{equation}\label{eq:A-B-C-at-zero}
A(0)=n,
\qquad
B(0)=n,
\qquad
C(0)=n,
\end{equation}
we have
\begin{equation}\label{eq:N-D-at-zero}
N(0)=n(n+n-2n)=0,
\qquad
D(0)=n^2-n^2=0.
\end{equation}
\end{proof}

\section{Balanced condition and the small-scale formula}

\begin{definition}[Balanced condition]\label{def:balanced-condition}
An \((n,k;\alpha,\beta,\gamma,\delta)\)-two-chunk cyclic metric space \(X_{n,k}(\alpha,\beta,\gamma,\delta)\) is said to satisfy the \emph{balanced condition} if
\begin{equation}\tag{\ensuremath{\star}}\label{eq:balanced-condition}
2\bigl(k\gamma+(n-k)\delta\bigr)=(n-1)(\alpha+\beta).
\end{equation}
A scaled family \(tX\) is called a \emph{balanced \(t\)-scaled \((n,k;\alpha,\beta,\gamma,\delta)\)-two-chunk cyclic metric space} if the underlying space \(X\) satisfies \eqref{eq:balanced-condition}.
\end{definition}

\begin{proposition}[First derivatives under the balanced condition]\label{prop:first-derivatives-balanced}
Let \(tX\) be a balanced \(t\)-scaled \((n,k;\alpha,\beta,\gamma,\delta)\)-two-chunk cyclic metric space.  Then
\begin{align}
N'(0)&=0,
\label{eq:N-prime-zero}\\
D'(0)&=0.
\label{eq:D-prime-zero}
\end{align}
\end{proposition}

\begin{proof}
Put
\begin{equation}\label{eq:sigma-one}
\sigma_1=k\gamma+(n-k)\delta.
\end{equation}
Then
\begin{align}
A'(0)&=-(n-1)\alpha,
&
B'(0)&=-(n-1)\beta,
&
C'(0)&=-\sigma_1.
\label{eq:first-derivatives-A-B-C}
\end{align}
Therefore
\begin{align}
N'(0)&=n\bigl(-(n-1)(\alpha+\beta)+2\sigma_1\bigr),
\label{eq:N-prime-computation}\\
D'(0)&=n\bigl(-(n-1)(\alpha+\beta)+2\sigma_1\bigr).
\label{eq:D-prime-computation}
\end{align}
Both quantities vanish exactly under \eqref{eq:balanced-condition}.
\end{proof}

\begin{proposition}[Two applications of l'Hopital's rule]\label{prop:lhopital}
Let \(tX\) be a balanced \(t\)-scaled \((n,k;\alpha,\beta,\gamma,\delta)\)-two-chunk cyclic metric space.  If
\begin{equation}\label{eq:D-second-nonzero}
D''(0)\neq 0,
\end{equation}
then
\begin{equation}\label{eq:lhopital-formula}
\lim_{t\to 0+}\Mag(tX)=
\frac{N''(0)}{D''(0)}.
\end{equation}
\end{proposition}

\begin{proof}
By Proposition \ref{prop:N-D-vanish}, both \(N(t)\) and \(D(t)\) vanish at \(t=0\).  By Proposition \ref{prop:first-derivatives-balanced}, their first derivatives also vanish at \(t=0\).  Since \(D''(0)\neq 0\), the functions \(D(t)\) and \(D'(t)\) are nonzero for all sufficiently small positive \(t\).  Since \(N(t)\) and \(D(t)\) are real analytic functions of \(t\), l'Hopital's rule applied twice gives
\begin{equation}\label{eq:lhopital-proof}
\lim_{t\to 0+}\frac{N(t)}{D(t)}
=
\lim_{t\to 0+}\frac{N'(t)}{D'(t)}
=
\lim_{t\to 0+}\frac{N''(t)}{D''(t)}
=
\frac{N''(0)}{D''(0)}.
\end{equation}
The last equality follows from the continuity of \(N''\) and \(D''\) at
\(0\), together with \(D''(0)\neq0\).  Using \eqref{eq:mag-N-D} proves
\eqref{eq:lhopital-formula}.
\end{proof}

\begin{corollary}[Explicit balanced small-scale formula]\label{cor:one-plus-formula}
Let \(tX\) be a balanced \(t\)-scaled \((n,k;\alpha,\beta,\gamma,\delta)\)-two-chunk cyclic metric space.  Define
\begin{equation}\label{eq:Omega-definition}
\Omega_{n,k}(\alpha,\beta,\gamma,\delta)
=2(n-1)(\alpha^2+\beta^2)-4k(n-k)(\gamma-\delta)^2.
\end{equation}
If \(\Omega_{n,k}(\alpha,\beta,\gamma,\delta)\neq 0\), then
\begin{equation}\label{eq:one-plus-formula}
\lim_{t\to 0+}\Mag(tX)
=
1+\frac{(n-1)^2(\alpha-\beta)^2}
{\Omega_{n,k}(\alpha,\beta,\gamma,\delta)}.
\end{equation}
\end{corollary}

\begin{proof}
Put
\begin{equation}\label{eq:sigma-two}
\sigma_2=k\gamma^2+(n-k)\delta^2.
\end{equation}
A direct calculation gives
\begin{align}
N''(0)&=n\bigl((n-1)(\alpha^2+\beta^2)-2\sigma_2\bigr),
\label{eq:N-second}\\
D''(0)&=n(n-1)(\alpha^2+\beta^2)+2(n-1)^2\alpha\beta
-2\sigma_1^2-2n\sigma_2.
\label{eq:D-second-general}
\end{align}
Under the balanced condition,
\begin{equation}\label{eq:sigma-one-balanced}
\sigma_1=\frac{(n-1)(\alpha+\beta)}{2}.
\end{equation}
Moreover,
\begin{equation}\label{eq:sigma-two-variance}
\sigma_2=\frac{\sigma_1^2}{n}+\frac{k(n-k)}{n}(\gamma-\delta)^2.
\end{equation}
Substituting \eqref{eq:sigma-one-balanced} and \eqref{eq:sigma-two-variance} into \eqref{eq:N-second} and \eqref{eq:D-second-general} yields
\begin{align}
D''(0)&=(n-1)(\alpha^2+\beta^2)-2k(n-k)(\gamma-\delta)^2,
\label{eq:D-second-simplified}\\
N''(0)&=D''(0)+\frac{(n-1)^2}{2}(\alpha-\beta)^2.
\label{eq:N-D-second-relation}
\end{align}
Since \(\Omega_{n,k}=2D''(0)\), Proposition \ref{prop:lhopital} gives
\begin{equation}\label{eq:one-plus-proof}
\lim_{t\to 0+}\Mag(tX)
=1+\frac{\frac{(n-1)^2}{2}(\alpha-\beta)^2}{D''(0)}
=1+\frac{(n-1)^2(\alpha-\beta)^2}{\Omega_{n,k}(\alpha,\beta,\gamma,\delta)}.
\end{equation}
\end{proof}

\section{Small-scale consequences}

\begin{proposition}[The cases \(n=3,4,5\)]\label{prop:n-3-4-5}
Let \(3\leq n\leq 5\), and let \(1\leq k\leq \lfloor n/2\rfloor\).  For every balanced \(t\)-scaled \((n,k;\alpha,\beta,\gamma,\delta)\)-two-chunk cyclic metric space for which the small-scale limit is finite,
\begin{equation}\label{eq:n-3-4-5-geq-one}
\lim_{t\to 0+}\Mag(tX)\geq 1.
\end{equation}
\end{proposition}

\begin{proof}
The possible pairs \((n,k)\) with \(3\leq n\leq 5\) and
\(1\leq k\leq \lfloor n/2\rfloor\) are
\begin{equation}\label{eq:small-n-all-pairs}
(3,1),\qquad (4,1),\qquad (4,2),\qquad (5,1),\qquad (5,2).
\end{equation}
We check these cases one by one, using only the triangle inequalities and the
balanced condition.

Set
\begin{equation}\label{eq:u-v-eta}
u=\max\{\alpha,\beta\},
\qquad
v=\min\{\alpha,\beta\},
\qquad
\eta=|\gamma-\delta|.
\end{equation}
The mixed triangle inequalities give
\begin{equation}\label{eq:eta-leq-v}
\eta\leq v.
\end{equation}
Recall that
\begin{equation}\label{eq:Omega-small-proof}
\Omega_{n,k}
=
2(n-1)(u^2+v^2)-4k(n-k)\eta^2.
\end{equation}
By Corollary \ref{cor:one-plus-formula}, it is enough to show that
\(\Omega_{n,k}\) cannot be negative.  Indeed, if \(\Omega_{n,k}>0\), then
\[
\lim_{t\to 0+}\Mag(tX)
=
1+\frac{(n-1)^2(\alpha-\beta)^2}{\Omega_{n,k}}
\geq 1.
\]

We first treat the three cases with \(k=1\).

For \((n,k)=(3,1)\), we have
\begin{equation}\label{eq:n3k1-Omega}
\Omega_{3,1}=4(u^2+v^2)-8\eta^2.
\end{equation}
If \(\Omega_{3,1}<0\), then
\begin{equation}\label{eq:n3k1-contradiction}
\eta^2>\frac{u^2+v^2}{2}\geq v^2,
\end{equation}
hence \(\eta>v\), contradicting \eqref{eq:eta-leq-v}.  Thus
\(\Omega_{3,1}\geq 0\).

For \((n,k)=(4,1)\), we have
\begin{equation}\label{eq:n4k1-Omega}
\Omega_{4,1}=6(u^2+v^2)-12\eta^2.
\end{equation}
If \(\Omega_{4,1}<0\), then again
\begin{equation}\label{eq:n4k1-contradiction}
\eta^2>\frac{u^2+v^2}{2}\geq v^2,
\end{equation}
so \(\eta>v\), contradicting \eqref{eq:eta-leq-v}.  Hence
\(\Omega_{4,1}\geq 0\).

For \((n,k)=(5,1)\), we have
\begin{equation}\label{eq:n5k1-Omega}
\Omega_{5,1}=8(u^2+v^2)-16\eta^2.
\end{equation}
If \(\Omega_{5,1}<0\), then
\begin{equation}\label{eq:n5k1-contradiction}
\eta^2>\frac{u^2+v^2}{2}\geq v^2,
\end{equation}
so \(\eta>v\), contradicting \eqref{eq:eta-leq-v}.  Thus
\(\Omega_{5,1}\geq 0\).

Next consider \((n,k)=(4,2)\).  The balanced condition is
\begin{equation}\label{eq:n4k2-balanced-sum}
\gamma+\delta=\frac{3}{4}(u+v).
\end{equation}
Since \(k=2\) and \(n-k=2\), both the \(\gamma\)-edges and the
\(\delta\)-edges occur in pairs.  Therefore the triangle inequalities force
\begin{equation}\label{eq:n4k2-lower-bounds}
u\leq 2\gamma,
\qquad
u\leq 2\delta,
\end{equation}
or equivalently
\begin{equation}\label{eq:n4k2-gamma-delta-lower}
\gamma,\delta\geq \frac{u}{2}.
\end{equation}
In particular,
\begin{equation}\label{eq:n4k2-necessary}
\frac{3}{4}(u+v)=\gamma+\delta\geq u,
\end{equation}
so a metric space can exist only if \(u\leq 3v\).  Moreover, under the lower
bounds \eqref{eq:n4k2-gamma-delta-lower}, the largest possible value of
\(|\gamma-\delta|\) occurs when the smaller of \(\gamma,\delta\) is \(u/2\).
Thus
\begin{equation}\label{eq:n4k2-eta-bound}
\eta
\leq
\frac{3}{4}(u+v)-u
=
\frac{3v-u}{4}.
\end{equation}
Now
\begin{equation}\label{eq:n4k2-Omega}
\Omega_{4,2}=6(u^2+v^2)-16\eta^2.
\end{equation}
If \(\Omega_{4,2}<0\), then
\begin{equation}\label{eq:n4k2-negative-condition}
\eta^2>\frac{3}{8}(u^2+v^2).
\end{equation}
On the other hand, \eqref{eq:n4k2-eta-bound} gives
\begin{equation}\label{eq:n4k2-upper-square}
\eta^2\leq \frac{(3v-u)^2}{16}.
\end{equation}
But, since \(u\geq v>0\),
\begin{equation}\label{eq:n4k2-elementary-positive}
6(u^2+v^2)-(3v-u)^2
=
5u^2+6uv-3v^2
>0.
\end{equation}
Equivalently,
\begin{equation}\label{eq:n4k2-contradiction}
\frac{(3v-u)^2}{16}
<
\frac{3}{8}(u^2+v^2),
\end{equation}
which contradicts \eqref{eq:n4k2-negative-condition}.  Hence
\(\Omega_{4,2}>0\).

Finally consider \((n,k)=(5,2)\).  The balanced condition is
\begin{equation}\label{eq:n5k2-balanced}
2\gamma+3\delta=2(u+v).
\end{equation}
Again, since \(k=2\) and \(n-k=3\), the triangle inequalities force
\begin{equation}\label{eq:n5k2-lower-bounds}
u\leq 2\gamma,
\qquad
u\leq 2\delta,
\end{equation}
so
\begin{equation}\label{eq:n5k2-gamma-delta-lower}
\gamma,\delta\geq \frac{u}{2}.
\end{equation}
Therefore
\begin{equation}\label{eq:n5k2-necessary}
2(u+v)=2\gamma+3\delta\geq \frac{5u}{2},
\end{equation}
and hence a metric space can exist only if
\begin{equation}\label{eq:n5k2-u-bound}
u\leq 4v.
\end{equation}

We now bound \(\eta=|\gamma-\delta|\) directly.  If \(\gamma\geq \delta\),
then \(\gamma=\delta+\eta\), and \eqref{eq:n5k2-balanced} gives
\begin{equation}\label{eq:n5k2-case-one-delta}
5\delta+2\eta=2(u+v).
\end{equation}
Since \(\delta\geq u/2\), we obtain
\begin{equation}\label{eq:n5k2-case-one-eta}
\eta\leq v-\frac{u}{4}
=
\frac{4v-u}{4}.
\end{equation}
If \(\delta\geq \gamma\), then \(\delta=\gamma+\eta\), and
\eqref{eq:n5k2-balanced} gives
\begin{equation}\label{eq:n5k2-case-two-gamma}
5\gamma+3\eta=2(u+v).
\end{equation}
Since \(\gamma\geq u/2\), we obtain
\begin{equation}\label{eq:n5k2-case-two-eta}
\eta\leq \frac{4v-u}{6}
\leq
\frac{4v-u}{4},
\end{equation}
where the last inequality uses \eqref{eq:n5k2-u-bound}.  Hence, in all cases,
\begin{equation}\label{eq:n5k2-eta-bound}
\eta\leq \frac{4v-u}{4}.
\end{equation}

Now
\begin{equation}\label{eq:n5k2-Omega}
\Omega_{5,2}=8(u^2+v^2)-24\eta^2.
\end{equation}
If \(\Omega_{5,2}<0\), then
\begin{equation}\label{eq:n5k2-negative-condition}
\eta^2>\frac{u^2+v^2}{3}.
\end{equation}
But \eqref{eq:n5k2-eta-bound} implies
\begin{equation}\label{eq:n5k2-upper-square}
\eta^2\leq \frac{(4v-u)^2}{16}.
\end{equation}
Since \(u\geq v>0\),
\begin{equation}\label{eq:n5k2-elementary-positive}
16(u^2+v^2)-3(4v-u)^2
=
13u^2+24uv-32v^2
>0.
\end{equation}
Equivalently,
\begin{equation}\label{eq:n5k2-contradiction}
\frac{(4v-u)^2}{16}
<
\frac{u^2+v^2}{3},
\end{equation}
contradicting \eqref{eq:n5k2-negative-condition}.  Hence
\(\Omega_{5,2}>0\).

We have shown, case by case, that \(\Omega_{n,k}<0\) is impossible for every
pair in \eqref{eq:small-n-all-pairs}.  Therefore, in all cases with
\(\Omega_{n,k}>0\), Corollary \ref{cor:one-plus-formula} gives
\[
\lim_{t\to 0+}\Mag(tX)\geq 1.
\]

It remains only to handle the possible equality case \(\Omega_{n,k}=0\).
The above argument shows that equality can occur only in the cases \(k=1\),
and then necessarily
\begin{equation}\label{eq:k1-equality-case}
u=v,
\qquad
\eta=v.
\end{equation}
Thus \(\alpha=\beta=v\), and hence \(A(t)=B(t)\).  The magnitude formula
\eqref{eq:magnitude-formula} becomes, after cancelling the common factor
\(A(t)-C(t)\),
\begin{equation}\label{eq:k1-degenerate-cancellation}
\Mag(tX)
=
\frac{2n(A(t)-C(t))}{(A(t)-C(t))(A(t)+C(t))}
=
\frac{2n}{A(t)+C(t)}.
\end{equation}
Since \(A(0)=C(0)=n\), we get
\begin{equation}\label{eq:k1-degenerate-limit-final}
\lim_{t\to 0+}\Mag(tX)
=
\frac{2n}{2n}
=
1.
\end{equation}
Thus every balanced \(t\)-scaled \((n,k;\alpha,\beta,\gamma,\delta)\)-two-chunk cyclic metric space with
\(3\leq n\leq 5\) has finite small-scale magnitude at least \(1\).
\end{proof}

\begin{proposition}[The case \(n=6\)]\label{prop:n-six}
Let \(n=6\), and let \(1\leq k\leq 3\).
\begin{enumerate}[label=\textup{(\arabic*)}]
\item For \(k=1\) and \(k=3\), there is no balanced \(t\)-scaled \((6,k;\alpha,\beta,\gamma,\delta)\)-two-chunk cyclic metric space with finite small-scale magnitude less than \(1\).
\item For \(k=2\), for every \(R\in\R\setminus\{1\}\), there exists a balanced \(t\)-scaled \((6,2;\alpha,\beta,\gamma,\delta)\)-two-chunk cyclic metric space such that
\begin{equation}\label{eq:n6k2-arbitrary-R}
\lim_{t\to 0+}\Mag(tX)=R.
\end{equation}
\item In particular, there exists a twelve-point finite metric space such that
\begin{equation}\label{eq:twelve-point-below-one}
\lim_{t\to 0+}\Mag(tX)<1.
\end{equation}
\end{enumerate}
\end{proposition}

\begin{proof}
Let \(u,v,\eta\) be as in \eqref{eq:u-v-eta}.

For \(k=1\), the estimate
\begin{equation}\label{eq:n6k1-Omega}
\Omega_{6,1}=10(u^2+v^2)-20\eta^2
\geq 10(u^2-v^2)
\end{equation}
shows that \(\Omega_{6,1}\geq 0\).  If \(\Omega_{6,1}>0\), then Corollary \ref{cor:one-plus-formula} gives a limit at least \(1\).  If \(\Omega_{6,1}=0\), then \(\alpha=\beta\), and the same direct expansion as in \eqref{eq:k1-degenerate-limit-final} gives limit \(1\).

For \(k=3\), the balanced condition gives
\begin{equation}\label{eq:n6k3-balanced}
\gamma+\delta=\frac{5}{6}(u+v).
\end{equation}
Since \(\gamma,\delta\geq u/2\),
\begin{equation}\label{eq:n6k3-eta-bound}
\eta\leq \gamma+\delta-u=\frac{5v-u}{6}.
\end{equation}
Thus
\begin{equation}\label{eq:n6k3-Omega-positive}
\Omega_{6,3}
=10(u^2+v^2)-36\eta^2
\geq 10(u^2+v^2)-(5v-u)^2>0.
\end{equation}
Again Corollary \ref{cor:one-plus-formula} implies that the finite small-scale limit is at least \(1\).  This proves \textup{(1)}.

It remains to prove \textup{(2)}.  We use two explicit one-parameter families.
The particular normalizations used below are not canonical; they are chosen to
make the balanced condition automatic and to keep the resulting one-variable
formulae transparent.

First, we impose the simplifying condition \(\gamma=\delta\).  If
\(\alpha=r\) and \(\beta=1\), then the balanced condition forces
\(\gamma=\delta=5(r+1)/12\).  Thus, for \(1\leq r\leq 5\), define
\begin{equation}\label{eq:n6-family-zero}
\alpha=r,
\qquad
\beta=1,
\qquad
\gamma=\delta=\frac{5(r+1)}{12}.
\end{equation}
This family is balanced, and Lemma \ref{lem:metric-conditions} shows that it is a metric family exactly in the displayed range.  Since \(\gamma=\delta\), Corollary \ref{cor:one-plus-formula} gives
\begin{equation}\label{eq:n6-L-zero}
L_0(r)
:=\lim_{t\to 0+}\Mag(tX)
=1+\frac{25(r-1)^2}{10(r^2+1)}.
\end{equation}
The function \(L_0\) is continuous and increasing on \([1,5]\), and
\begin{equation}\label{eq:n6-L-zero-endpoints}
L_0(1)=1,
\qquad
L_0(5)=\frac{33}{13}.
\end{equation}
Hence this family realizes every value in \([1,33/13]\).

Second, for \(1\leq r\leq 5\), define
\begin{equation}\label{eq:n6-family-one}
\alpha=r,
\qquad
\beta=1,
\qquad
\delta=\frac{r}{2},
\qquad
\gamma=\frac{r+5}{4}.
\end{equation}
This family is also balanced and metric by Lemma \ref{lem:metric-conditions}.  Here
\begin{equation}\label{eq:n6-family-one-eta}
\gamma-
\delta=\frac{5-r}{4},
\end{equation}
and Corollary \ref{cor:one-plus-formula} gives
\begin{equation}\label{eq:n6-L-one}
L_1(r)
:=\lim_{t\to 0+}\Mag(tX)
=1+\frac{25(r-1)^2}{4(2r^2+5r-10)}.
\end{equation}
The denominator in \eqref{eq:n6-L-one} vanishes at
\begin{equation}\label{eq:n6-root}
r_0=\frac{\sqrt{105}-5}{4}\in(1,5).
\end{equation}
Moreover,
\begin{align}
\lim_{r\to r_0-}L_1(r)&=-\infty,
\label{eq:n6-negative-infinity}\\
\lim_{r\to r_0+}L_1(r)&=+\infty,
\label{eq:n6-positive-infinity}\\
L_1(1)&=1,
\label{eq:n6-L-one-at-one}\\
L_1(5)&=\frac{33}{13}.
\label{eq:n6-L-one-at-five}
\end{align}
By continuity, \(L_1\) realizes every value below \(1\) on \([1,r_0)\), and every value at least \(33/13\) on \((r_0,5]\).  Together with \eqref{eq:n6-L-zero}, this realizes every value in \(\R\setminus\{1\}\).  This proves \textup{(2)}.

Finally, \textup{(3)} follows from \textup{(2)} by choosing any \(R<1\).  The concrete matrix in \eqref{eq:intro-matrix} gives the explicit value \(R=44/59\).
\end{proof}

\begin{proposition}[The cases \(n\geq 7\)]\label{prop:n-geq-seven}
Let \(n\geq 7\) and let
\begin{equation}\label{eq:n-geq-seven-k-condition}
2\leq k\leq \left\lfloor\frac{n}{2}\right\rfloor.
\end{equation}
Then for every \(R\in\R\setminus\{1\}\), there exists a balanced \(t\)-scaled \((n,k;\alpha,\beta,\gamma,\delta)\)-two-chunk cyclic metric space such that
\begin{equation}\label{eq:n-geq-seven-arbitrary-R}
\lim_{t\to 0+}\Mag(tX)=R.
\end{equation}
\end{proposition}

\begin{proof}
Put
\begin{equation}\label{eq:m-q-definition}
m=n-1,
\qquad
q=n-k.
\end{equation}
Since \(n\geq 7\) and \(2\leq k\leq \lfloor n/2\rfloor\), one can choose a real number \(\theta\) satisfying
\begin{equation}\label{eq:theta-choice}
\sqrt{\frac{m}{kq}}<\theta<\min\left\{1,\frac{n-2}{2k}\right\}.
\end{equation}
Indeed, \(kq>m\), and the inequality
\begin{equation}\label{eq:theta-interval-nonempty}
\frac{m}{kq}<\frac{(n-2)^2}{4k^2}
\end{equation}
follows from \(k\leq n/2\), \(q\geq n/2\), and \((n-2)^2>4(n-1)\) for \(n\geq 7\).

For \(r>1\) close to \(1\) and for \(0\leq h\leq \theta\), define
\begin{align}
\beta&=1,
\label{eq:general-family-beta}\\
\alpha&=r,
\label{eq:general-family-alpha}\\
\delta(r,h)&=\frac{(n-1)(r+1)}{2n}-\frac{k}{n}h,
\label{eq:general-family-delta}\\
\gamma(r,h)&=\delta(r,h)+h.
\label{eq:general-family-gamma}
\end{align}
Then
\begin{equation}\label{eq:general-family-balanced}
k\gamma(r,h)+(n-k)\delta(r,h)=\frac{(n-1)(r+1)}{2},
\end{equation}
so the balanced condition holds.  By the strict upper bound in \eqref{eq:theta-choice}, the inequalities in Lemma \ref{lem:metric-conditions} hold for all \(0\leq h\leq\theta\), after restricting \(r\) to a sufficiently small interval \((1,1+\varepsilon)\).  Hence \eqref{eq:general-family-beta}--\eqref{eq:general-family-gamma} define balanced \((n,k;\alpha,\beta,\gamma,\delta)\)-two-chunk cyclic metric spaces.

For this family, Corollary \ref{cor:one-plus-formula} gives
\begin{equation}\label{eq:general-family-limit}
L(r,h)
:=\lim_{t\to 0+}\Mag(tX)
=1+\frac{m^2(r-1)^2}{2m(r^2+1)-4kq h^2},
\end{equation}
whenever the denominator is nonzero.  For \(r\) sufficiently close to \(1\), the denominator in \eqref{eq:general-family-limit} is positive at \(h=0\), while it is negative at \(h=\theta\), because
\begin{equation}\label{eq:general-denominator-signs}
2m(1^2+1)-4kq\theta^2=4m-4kq\theta^2<0.
\end{equation}
Thus the denominator crosses zero as \(h\) varies from \(0\) to \(\theta\).  On the positive side of the crossing, \(L(r,h)\) tends to \(+\infty\); on the negative side, it tends to \(-\infty\).  Also,
\begin{align}
L(r,0)&=1+\frac{m(r-1)^2}{2(r^2+1)}\longrightarrow 1
\quad \text{as } r\to 1+,
\label{eq:general-L-at-zero}\\
L(r,\theta)&\longrightarrow 1
\quad \text{as } r\to 1+,
\label{eq:general-L-at-theta}
\end{align}
with \(L(r,\theta)<1\) for \(r\) sufficiently close to \(1\).

Now let \(R>1\).  Choose \(r>1\) sufficiently close to \(1\) so that \(L(r,0)<R\).  Since \(L(r,h)\to +\infty\) at the positive side of the zero of the denominator, the intermediate value theorem gives an \(h\in[0,\theta]\) with \(L(r,h)=R\).  Similarly, if \(R<1\), choose \(r>1\) sufficiently close to \(1\) so that \(L(r,\theta)>R\).  Since \(L(r,h)\to -\infty\) at the negative side of the zero of the denominator, the intermediate value theorem again gives an \(h\in[0,\theta]\) with \(L(r,h)=R\).  This proves the proposition.
\end{proof}

\section{Example revisited}

\begin{example}[The twelve-point example]\label{ex:twelve-point}
The matrix in \eqref{eq:intro-matrix} is
\begin{equation}\label{eq:example-as-two-chunk}
D_{6,2}\left(\frac{7}{3},\frac{29}{15},3,\frac{7}{6}\right).
\end{equation}
The triangle inequalities follow from Lemma \ref{lem:metric-conditions}, since
\begin{align}
\max\left\{\frac{7}{3},\frac{29}{15}\right\}&=\frac{7}{3}\leq 2\cdot 3,
\label{eq:example-triangle-one}\\
\max\left\{\frac{7}{3},\frac{29}{15}\right\}&=\frac{7}{3}=2\cdot \frac{7}{6},
\label{eq:example-triangle-two}\\
\max\left\{\frac{7}{3},\frac{29}{15}\right\}&=\frac{7}{3}\leq 3+\frac{7}{6},
\label{eq:example-triangle-three}\\
\left|3-\frac{7}{6}\right|&=\frac{11}{6}<\frac{29}{15}
=\min\left\{\frac{7}{3},\frac{29}{15}\right\}.
\label{eq:example-triangle-four}
\end{align}
It is balanced because
\begin{equation}\label{eq:example-balanced}
2\left(2\cdot 3+4\cdot \frac{7}{6}\right)
=\frac{64}{3}
=5\left(\frac{7}{3}+\frac{29}{15}\right).
\end{equation}
Moreover,
\begin{align}
\Omega_{6,2}
&=10\left(\left(\frac{7}{3}\right)^2+
\left(\frac{29}{15}\right)^2\right)
-32\left(3-\frac{7}{6}\right)^2
\label{eq:example-Omega-computation}\\
&=-\frac{236}{15},
\label{eq:example-Omega-value}
\end{align}
and
\begin{equation}\label{eq:example-numerator-computation}
(6-1)^2\left(\frac{7}{3}-\frac{29}{15}\right)^2=4.
\end{equation}
Therefore Corollary \ref{cor:one-plus-formula} gives
\begin{equation}\label{eq:example-final-limit}
\lim_{t\to 0+}\Mag(tX)
=1+\frac{4}{-236/15}
=\frac{44}{59}<1.
\end{equation}
\end{example}

\section*{Acknowledgements}

The author used LLM as an auxiliary tool during the final stage of this
work, mainly for brainstorming, checking algebraic manipulations, and improving
the exposition. The research direction and mathematical content were developed,
reviewed, and verified by the author, who takes full responsibility for the
paper.

\end{document}